\documentclass[12pt,longbibliography]{article}

\usepackage[utf8]{inputenc} 
\usepackage[T1]{fontenc}

\usepackage{graphicx}
\usepackage{amsmath}
\usepackage{amsfonts}
\usepackage{mathrsfs}           
\usepackage{amsthm}
\usepackage{amssymb}
\usepackage{color}
\usepackage{enumerate}
\usepackage{comment}
\usepackage{bbold} 

\usepackage[a4paper, margin=2.7cm]{geometry}

\usepackage[colorlinks=true,linkcolor=blue,citecolor=blue,urlcolor=blue,breaklinks]{hyperref}

\usepackage{bbm}

\usepackage{breakurl}
\usepackage{url}

\def\C{\mathbb{C}}
\def\N{\mathbb{N}}
\def\R{\mathbb{R}}

\def\d{\,\mathrm{d}}

\newtheorem{thm}{Theorem}[section]
\newtheorem{cor}[thm]{Corollary}
\newtheorem{lem}[thm]{Lemma}
\newtheorem{prp}[thm]{Proposition}

\theoremstyle{definition}
\newtheorem{dfn}[thm]{Definition}
\theoremstyle{remark}
\newtheorem{rem}[thm]{Remark}
\theoremstyle{example}

\title{Ultracontractivity of Heat semigroups in $\mathrm{L}^{2}\left( \Omega \right)$ with non-local Robin boundary conditions
using Nash's inequality}

\author{Christoph Schwerdt$^1$}
\date{
	$^1$ Institute of Mathematics, University of Rostock,\\
	Ulmenstra\ss e 69, 18 057 Rostock, Germany \\
	\ \\
	\today
}

\begin{document}

\maketitle

\begin{abstract}
We study heat equations $\frac{\partial u}{\partial t} - \operatorname{div} \left( A \nabla u \right) = 0$ on bounded
Lipschitz domains $\Omega$ in $\R^{d}$ for $d>2$, where $-\operatorname{div} \left( A \nabla \cdot \right)$ is a 
second-order uniformly elliptic operator with generalised Robin boundary conditions. These boundary conditions 
are formally given by $\nu \cdot A \nabla u + Bu = 0$ where $\nu$ is the outer unit normal on $\partial\Omega$ and 
$B \in \mathcal{L} \left( \mathrm{L}^{2}\left( \partial \Omega  \right) \right)$ is a general operator which
is allowed to destroy the positivity preserving property of the solution semigroup. Ultracontractivity of
the solution semigroup is shown by using Nash's inequality on the Sobolev space $H^{1}( \Omega )$. 
\end{abstract}

\tableofcontents

\section{Introduction}

Let $\Omega \subset \R^{d}$ for $d \in \N$ be a bounded Lipschitz domain and $A \colon \Omega \to \R^{d \times d}$ be
a matrix-valued function such that every coefficient $a_{ij} \colon \Omega \to \R$ is a bounded and measurable function 
on $\Omega$. Furthermore let $A$ be uniformly elliptic on $\Omega$, i.e.
$$
\exists \alpha > 0 \ \forall \xi \in \C^{d}, x \in \Omega \ : \ 
\Re \left(  \left( A(x) \xi \right)^{T} \overline{\xi} \right) \geq \alpha \left| \xi \right|^{2}. 
$$ 
We study solutions $u = u(t,x)$ to the second-order parabolic equation
\begin{equation}\label{Generalised_Heat_Equation}
\frac{\partial u}{\partial t} - \operatorname{div} \left( A \nabla u \right) = 0 \ \text{ in } \ (0,\infty) \times \Omega 
\end{equation}
with a generalised Robin boundary condtion formally given by
\begin{equation}\label{Robin_Condition}
\nu \cdot A \nabla u + B \gamma(u) = 0 \ \text{ on } \ (0,\infty) \times \partial \Omega,
\end{equation}
where $\nu$ is the outer unit normal on $\partial\Omega$ and $B$ is a linear operator on 
$\mathrm{L}^{2}\left( \partial \Omega  \right)$. In this article $a(B)$ always is a sesqulinear form 
in $\mathrm{L}^{2}\left( \Omega \right)$ formally given by
\begin{equation}\label{definition_form}
a(B)(u,v) = \int_{\Omega} A \nabla u \cdot \overline{\nabla v} \d x + 
\int_{\partial \Omega} B \gamma(u) \ \overline{\gamma(v)} \d \sigma(x)
\end{equation}
for $u,v \in D \left( a(B) \right) = H^{1}(\Omega)$ where $\gamma$ is a trace operator mapping $H^{1}(\Omega)$
into $\mathrm{L}^{2}\left( \partial \Omega \right)$. 
Let $L(B)$ be the associated operator to $a(B)$ such that $-L(B)$ generates a 
$C_{0}$-semigroup $\left( \mathrm{e}^{-tL(B)} \right)_{t \geq 0}$ in $\mathrm{L}^{2}\left( \Omega \right)$. 
Then
$$
u(t) = \mathrm{e}^{-tL(B)} u_{0}
$$
is a solution to (\ref{Generalised_Heat_Equation}) taking (\ref{Robin_Condition}) into account in the sense of Subsection 2.4
in the article \cite{Glueck_Mui_2026} by Jochen Glück and Jonathan Mui.
For details about previous articles on the topic, more information how to model heat equations with non-local boundary conditions
as well as an implication of an eventual positivity of the solution 
semigroup please see Section 1 of  \cite{Glueck_Mui_2026}. 
Here we focus on one of their main results in form of Theorem 3.2 on page 10 which reads:

\begin{thm}\label{Theorem_Jochen}
Let $\Omega$ and $A$ be as described above. Furthermore let 
$B \in \mathcal{L}\left( \mathrm{L}^{2}\left( \partial \Omega \right) \right)$ act boundedly on 
$\mathrm{L}^{1}\left( \partial \Omega \right)$ and on $\mathrm{L}^{\infty}\left( \partial \Omega \right)$.
Then the solution semigroup operators $\mathrm{e}^{-tL(B)}$ are ultracontractive in $\mathrm{L}^{2}\left( \Omega \right)$
at any time $t \in (0,1]$, i.e.
$$
\exists \ C, \mu > 0 \ \forall u \in  \mathrm{L}^{2}\left( \Omega \right) \ : \ 
\left\| \mathrm{e}^{-tL(B)} u \right\|_{ \mathrm{L}^{\infty}\left( \Omega \right)}
\leq C t^{-\frac{\mu}{4}} \left\| u \right\|_{ \mathrm{L}^{2}\left( \Omega \right)} 
$$
is true for any time $t \in (0,1]$
\end{thm}

\begin{rem} \
\begin{enumerate}[i)]
\item Note that $B$ is defined on $\mathrm{L}^{\infty}\left( \partial \Omega \right)$ since 
$\mathrm{L}^{\infty}\left( \partial \Omega \right) \subset \mathrm{L}^{2}\left( \partial \Omega \right)$ is implied
by $\Omega$ being a bounded Lipschitz domain. Demanding $B$ to act boundedly on 
$\mathrm{L}^{\infty}\left( \partial \Omega \right)$ means
that $Bu$ is contained in the subspace 
$\mathrm{L}^{\infty}\left( \partial \Omega \right) \subset \mathrm{L}^{2}\left( \partial \Omega \right)$ 
for every $u \in \mathrm{L}^{\infty}\left( \partial \Omega \right)$ and furthermore that there exists a constant 
$C>0$ such that
$$
\left\| Bu \right\|_{\mathrm{L}^{\infty}\left( \partial \Omega \right)} \ \leq \ C 
\left\| u \right\|_{\mathrm{L}^{\infty}\left( \partial \Omega \right)}
$$
is true for every $u \in \mathrm{L}^{\infty}\left( \partial \Omega \right)$.
\item Furthermore $\mathrm{L}^{2}\left( \partial \Omega \right)$ is contained in $\mathrm{L}^{1}\left( \partial \Omega \right)$
since $\Omega$ is a bounded Lipschitz domain. Hence $Bu$ is contained in $\mathrm{L}^{1}\left( \partial \Omega \right)$ 
for every $u \in \mathrm{L}^{2}\left( \partial \Omega \right)$. Demanding $B$ to act boundedly on 
$\mathrm{L}^{1}\left( \partial \Omega \right)$ means that there exists  a constant $\tilde{C}>0$ such that
$$
\left\| Bu \right\|_{\mathrm{L}^{1}\left( \partial \Omega \right)} \ \leq \ \tilde{C} 
\left\| u \right\|_{\mathrm{L}^{1}\left( \partial \Omega \right)}
$$ 
is true for every $u \in \mathrm{L}^{2}\left( \partial \Omega \right)$. Then there exists a unique extension of $B$
in $\mathcal{L}\left( \mathrm{L}^{1}\left( \partial \Omega \right) \right)$ since $\mathrm{L}^{2}\left( \partial \Omega \right)$
is a dense subset in $\mathrm{L}^{1}\left( \partial \Omega \right)$.\\
\end{enumerate}
\end{rem}

The main argument in \cite{Glueck_Mui_2026} is to gradually construct an operator 
$B_{2} \in \mathcal{L}\left( \mathrm{L}^{2}\left( \partial \Omega \right) \right)$ such that $-B_{2}$ is positive
and satisfies
$$
\left| Bu \right| \ \leq \ -B_{2} |u|
$$
pointwise almost everywhere on $\partial \Omega$
for every $u \in \mathrm{L}^{2}\left( \partial \Omega \right)$. Furthermore $-L(B_{2})$ is the generator of another $C_{0}$-semigroup
$\left( \mathrm{e}^{-tL(B_{2})} \right)_{t \geq 0}$ such that $\mathrm{e}^{-tL(B_{2})}$ dominates $\mathrm{e}^{-tL(B)}$ at
every time $t > 0$, i.e.
\begin{equation}\label{domination_argument}
\left| \mathrm{e}^{-tL(B)} u \right| \ \leq \ \mathrm{e}^{-tL(B_{2})} |u|
\end{equation}
is true almost everywhere in $\Omega$ for every $u \in \mathrm{L}^{2}\left( \Omega \right)$. Please compare with
Step 4 on page 13 in \cite{Glueck_Mui_2026}. The construction of $B_{2}$ allows to utilize the results 
of Proposition 3.4 and Lemma 3.5 both on page 11 to conclude that $\mathrm{e}^{-tL(B_{2})}$ restricts to a bounded
operator on $\mathrm{L}^{\infty}\left( \Omega \right)$ which is crucial to prove the ultracontractivity of
 $\mathrm{e}^{-tL(B)}$  in $\mathrm{L}^{2}\left( \Omega \right)$ using (\ref{domination_argument}).\\

In this article we put more focus on the ellipticity of $A$ on $\Omega$. Note that for $B$ being a bounded in 
$\mathrm{L}^{2}\left( \partial \Omega \right)$ the domain $D\left( a(B) \right)$ is defined as the 
Sobolev space $H^{1} \left( \Omega \right)$ in $\mathrm{L}^{2}\left( \Omega \right)$. Furthermore notice that if the adjoint 
operators $\mathrm{e}^{-tL(B)^{\ast}}$ in $\mathrm{L}^{2}\left( \Omega \right)$ are continuous from 
$\mathrm{L}^{1}\left( \Omega \right)$ to 
$\mathrm{L}^{2}\left( \Omega \right)$ for every time $t > 0$, then an ultracontractivity of the original
solution semigroup operators $\mathrm{e}^{-tL(B)}$ in $\mathrm{L}^{2}\left( \Omega \right)$ is implied. 
Nash's inequality beautifully connects all of the mentioned function
spaces $\mathrm{L}^{1}\left( \Omega \right)$, $\mathrm{L}^{2}\left( \Omega \right)$ and $H^{1} \left( \Omega \right)$.
Using Nash's inequality as an alternative argumentation offers the same result as Theorem \ref{Theorem_Jochen}. 
It reads as follows:

\begin{thm}
Again let $\Omega$ and $A$ be as above. Furthermore let 
$B \in \mathcal{L}\left( \mathrm{L}^{2}\left( \partial \Omega \right) \right)$.
Suppose there exists a postive operator $B_{pos} \in \mathcal{L} \left( \mathrm{L}^{2}( \partial \Omega ) \right)$ that 
satisfies
$$
\left| B u \right| \ \leq \ B_{pos} \left| u \right|
$$
for every $u \in \mathrm{L}^{2}( \partial \Omega )$. Then $\mathrm{e}^{-tL(B)}$ are ultracontractive 
operators in $\mathrm{L}^{2}\left( \Omega \right)$ at any time $t > 0$. 
In particular there are constants $C>0$ and $\mu_{0} > 0$ satisfying
$$
\left\| \mathrm{e}^{-tL(B)} u \right\|_{\mathrm{L}^{\infty} \left( \Omega \right) } \leq C t^{-\frac{d}{4}} \mathrm{e}^{t \mu_{0}} 
\left\| u \right\|_{\mathrm{L}^{2} \left( \Omega \right) }
$$
for every time $t > 0$ and any $u \in \mathrm{L}^{2}\left( \Omega \right)$.
\end{thm}

\begin{rem} 
If $B \in \mathcal{L}\left( \mathrm{L}^{2}\left( \partial \Omega \right) \right)$ acts boundedly on 
$\mathrm{L}^{1}\left( \partial \Omega \right)$ and on $\mathrm{L}^{\infty}\left( \partial \Omega \right)$, then
the existence of such an operator $B_{pos} \in \mathcal{L} \left( \mathrm{L}^{2}( \partial \Omega ) \right)$ with
$$
\left| B u \right| \ \leq \ B_{pos} \left| u \right|
$$
for every $u \in \mathrm{L}^{2}( \partial \Omega )$ is implied by Step 1 on page 12 in \cite{Glueck_Mui_2026}. 
\end{rem}

Finally let us give some motivation by refering to the second main result in \cite{Glueck_Mui_2026} in form of 
Theorem 4.4 on page 18 which reads:

\begin{thm}\label{Theorem_Jochen_2}
Let $\Omega$ and $A$ be as above. Furthermore let 
$B \in \mathcal{L}\left( \mathrm{L}^{2}\left( \partial \Omega \right) \right)$ act boundedly on 
$\mathrm{L}^{1}\left( \partial \Omega \right)$ and on $\mathrm{L}^{\infty}\left( \partial \Omega \right)$. Furthermore
let $B + B^{\ast}$ be positive semi-definte and $B 1_{\partial \Omega} = 0$. Then there exists $t_{0} \geq 0$ and
a constant $\delta > 0$ such that 
$$
\mathrm{e}^{-tL_{B}} u \ \geq \ \delta \left( \int_{\Omega} u \d x \right)
$$
is true almost everywhere in $\Omega$ for every $t \geq t_{0}$ and every $0 \leq u \in \mathrm{L}^{2} \left( \Omega \right)$.
In particular, $\left( \mathrm{e}^{-tL_{B}} \right)_{t \geq 0}$ is uniformly eventually positive.
\end{thm}

We see on page 18 in \cite{Glueck_Mui_2026} that the only reason for demanding 
$B \in \mathcal{L}\left( \mathrm{L}^{2}\left( \partial \Omega \right) \right)$ 
to act boundedly on $\mathrm{L}^{1}\left( \partial \Omega \right)$ and on $\mathrm{L}^{\infty}\left( \partial \Omega \right)$
is to imply an ultracontractivity of $\mathrm{e}^{-tL_{B}}$.

\section{Preparations}\label{Preparations}

Let $\Omega \subset \R^{d}$ for $d \in \N$ be a bounded Lipschitz domain and $A \colon \Omega \to \R^{d \times d}$ be
a matrix-valued function such that every coefficient $a_{ij} \colon \Omega \to \R$ is a bounded and measurable function 
on $\Omega$. Furthermore let $A$ be uniformly elliptic on $\Omega$, i.e.
$$
\exists \alpha > 0 \ \forall \xi \in \C^{d}, x \in \Omega \ : \ \Re \left(  \left( A(x) \xi \right)^{T} \overline{\xi} \right) 
\geq \alpha \left| \xi \right|^{2}. 
$$
Let the boundary condition (\ref{Robin_Condition}) be characterized by a bounded operator 
$B$ in $ \mathrm{L}^{2}\left( \partial \Omega \right)$ and a continuous trace operator 
$\gamma \colon H^{1}(\Omega) \to \mathrm{L}^{2}(\partial \Omega)$. Similar to \cite{Glueck_Mui_2026} we also cite
an important result of \cite{Gesztesy_Mitrea_2008} in form of Lemma 2.5 on page 10.

\begin{lem}\label{gamma}
For every $\varepsilon > 0$ there exists a constant $\beta(\varepsilon) > 0$ such that
$$
\| \gamma(u) \|_{\mathrm{L}^{2}(\partial \Omega)} \ \leq \ \varepsilon \| \nabla u \|_{\mathrm{L}^{2}(\Omega)^{d}}^{2}
+ \beta(\varepsilon) \| u \|_{\mathrm{L}^{2}(\Omega)}^{2}
$$
is true for every $u \in H^{1}(\Omega)$.
\end{lem} 

Therefore the following corollary is a simple implication.

\begin{cor}\label{Operator_B}
For every $\varepsilon > 0$ there exists a constant $\beta(\varepsilon) > 0$ such that
\begin{align*}
& \Re \int_{\Omega} A \nabla u \cdot \overline{\nabla u} \d x + 
\Re \int_{\partial \Omega} B\gamma(u) \overline{\gamma(u)} \d \sigma(x) & \\
& \geq \ \left( \alpha - \varepsilon \| B \|_{\mathrm{L}^{2} \to \mathrm{L}^{2}} \right)  \| \nabla u \|_{\mathrm{L}^{2}(\Omega)^{d}}^{2}
- \beta(\varepsilon) \| B \|_{\mathrm{L}^{2} \to \mathrm{L}^{2}}   \| u \|_{\mathrm{L}^{2}(\Omega)}^{2} &
\end{align*}
is true for every $u \in H^{1}(\Omega)$.
\end{cor}

\begin{proof}
Let $\varepsilon > 0$ be arbitrary but fixed. First we use Lemma \ref{gamma} to argue that
\begin{align*}
- \Re \int_{\partial \Omega} B\gamma(u) \overline{\gamma(u)} \d \sigma(x) & 
\leq \left| \int_{\partial \Omega} B\gamma(u) \overline{\gamma(u)} \d \sigma(x) \right| \leq 
\| B \|_{\mathrm{L}^{2} \to \mathrm{L}^{2}} \| \gamma(u) \|_{\mathrm{L}^{2}(\partial \Omega)}^{2} & \\
\ \\
& \leq \varepsilon \| B \|_{\mathrm{L}^{2} \to \mathrm{L}^{2}}  \| \nabla u \|_{\mathrm{L}^{2}(\Omega)^{d}}^{2}
+ \beta(\varepsilon) \| B \|_{\mathrm{L}^{2} \to \mathrm{L}^{2}}   \| u \|_{\mathrm{L}^{2}(\Omega)}^{2} &
\end{align*}
is true for every $u \in H^{1}(\Omega)$. Furthermore 
$$
\Re \int_{\Omega} A \nabla u \cdot \overline{\nabla u} \d x \ \geq \ \alpha \| \nabla u \|_{\mathrm{L}^{2}(\Omega)^{d}}^{2}
$$
holds for $u \in H^{1}(\Omega)$ which proves the claim.
\end{proof}

\subsection{The $C_{0}$-semigroup $\mathrm{e}^{-tL_{B}}$ in $\mathrm{L}^{2}\left( \Omega \right)$
for $B \in \mathcal{L}\left( \mathrm{L}^{2}\left( \partial \Omega \right) \right)$} \label{form_a}

We set $\varepsilon_{0} = \frac{\alpha}{2}\| B \|_{\mathrm{L}^{2} \to \mathrm{L}^{2}}^{-1}$ and 
$\lambda_{0} \geq \beta(\varepsilon_{0}) \| B \|_{\mathrm{L}^{2} \to \mathrm{L}^{2}}$ arbitrary but fixed
both in the sense of Corollary \ref{Operator_B}. 

\begin{dfn}
Define an auxiliary sesquilinear form $\tilde{a}(B)$ in $\mathrm{L}^{2}\left( \Omega \right)$ by
$$
\tilde{a}(B)(u,v) = \int_{\Omega} A \nabla u \cdot \overline{\nabla v} \d x + 
\int_{\partial \Omega} B \gamma(u) \ \overline{\gamma(v)} \d \sigma(x) + \lambda_{0} \int_{\Omega} u \ \overline{v} \d x
$$
for $u,v \in D \left( \tilde{a}(B) \right) = H^{1}\left( \Omega \right)$ and the form $a(B)$ by
\begin{equation}\label{form_a_B}
a(B)(u,v) = \int_{\Omega} A \nabla u \cdot \overline{\nabla v} \d x + 
\int_{\partial \Omega} B \gamma(u) \ \overline{\gamma(v)} \d \sigma(x)
\end{equation}
for $u,v \in D \left( a(B) \right) = D \left( \tilde{a}(B) \right)$.
\end{dfn}

\begin{rem} \label{a(B)_Nash}
While $a(B)$ might not be accretive as demanded in Appendix \ref{results_by_Ouhabaz}, 
the auxiliary form $\tilde{a}(B)$ on the other side is due to Corollary \ref{Operator_B}
since
\begin{equation}
\Re \tilde{a}(B)(u,u) \geq  \frac{\alpha}{2} \| \nabla u \|_{\mathrm{L}^{2}(\Omega)^{d}}^{2} \geq 0
\end{equation}
holds for every $u \in D \left( \tilde{a}(B) \right)$. 
\end{rem}

Let us show the continuity and 
closedness of the auxiliary form $\tilde{a}(B)$ next. We define the induced norm of $\tilde{a}(B)$ by
$$
\| u \|_{\tilde{a}(B)} = \left( \Re \tilde{a}(B)(u,u) +  \| u \|_{\mathrm{L}^{2}\left( \Omega \right)}^{2} \right)^{\frac{1}{2}}
$$
for $u \in D \left( \tilde{a}(B) \right)$. Of course
\begin{equation}\label{H1_inequ}
\| u \|_{\tilde{a}(B)} \ = \ \left( \Re \tilde{a}(B)(u,u) +  \| u \|_{\mathrm{L}^{2}\left( \Omega \right)}^{2} \right)^{\frac{1}{2}}
\ \geq \ \min \left( \frac{\alpha}{2}, 1 \right)^{1/2} \| u \|_{H^{1}(\Omega)}
\end{equation}
is an obvious implication.

\begin{lem}\label{form_properties}
The form $\tilde{a}(B)$ is continuous and closed.
\end{lem}

\begin{proof} \
\begin{enumerate}[i.)]
\item Let us prove the continuity of $\tilde{a}(B)$ first. Note that 
$\| u \|_{\mathrm{L}^{2}\left( \Omega \right)} \leq \| u \|_{\tilde{a}(B)}$
is true for every $u \in D \left( \tilde{a}(B) \right)$. For $u,v \in D \left( \tilde{a}(B) \right)$ we easily conclude that
$$
\left| \tilde{a}(B)(u,v) \right| \leq \left| \langle A \nabla u, \nabla v \rangle_{\mathrm{L}^{2}\left( \Omega \right)^{d}}  \right|
+ \left| \langle B \gamma(u), \gamma(v) \rangle_{\mathrm{L}^{2}\left( \partial \Omega \right)} \right| 
+ \lambda_{0} \left| \langle u, v \rangle \right| 
$$
is true. Now we focus on 
\begin{align*}
\left| \langle A \nabla u, \nabla v \rangle_{\mathrm{L}^{2}\left( \Omega \right)^{d}}  \right| 
& = \left| \int_{\Omega} \left( A(x) (\nabla u)(x) \right)^{T}  \overline{(\nabla v)(x)} \d x  \right| & \\
& \leq \sum_{l=1}^{d} \sum_{k=1}^{d} \int_{\Omega} \left| a_{lk}(x) \right| \ \left| (\partial_{k}u)(x) \right| 
\ \left| (\partial_{l}u)(x) \right| \ \d x & \\
& \leq \max_{l,k} \left\| a_{l,k} \right\|_{\mathrm{L}^{\infty}\left( \Omega \right)}  \sum_{l=1}^{d} \sum_{k=1}^{d}
\left\| u \right\|_{H^{1}\left( \Omega \right)} \left\| v \right\|_{H^{1}\left( \Omega \right)}  & \\
& \leq d^{2}  \left\| A \right\|_{\mathrm{L}^{\infty}\left( \Omega \right)^{d \times d}} 
\left\| u \right\|_{H^{1}\left( \Omega \right)} \left\| v \right\|_{H^{1}\left( \Omega \right)}.  & 
\end{align*}
Therefore we conclude to
\begin{align*}
\left| \tilde{a}_{B}(u,v) \right| & \leq \left| \langle A \nabla u, \nabla v \rangle_{\mathrm{L}^{2}\left( \Omega \right)^{d}}  \right|
+ \left| \langle B \gamma(u), \gamma(v) \rangle_{\mathrm{L}^{2}\left( \partial \Omega \right)} \right| 
+ \lambda_{0} \left| \langle u, v \rangle \right| & \\
\ \\
& \leq d^{2} \left\| A \right\|_{\mathrm{L}^{\infty}\left( \Omega \right)^{d \times d}}  
\left\| u \right\|_{H^{1}\left( \Omega \right)} \left\| v \right\|_{H^{1}\left( \Omega \right)} 
+ \| B \|_{\mathrm{L}^{2} \to \mathrm{L}^{2}} \ 
\| \gamma (u) \|_{\mathrm{L}^{2}\left( \partial \Omega \right)} \ 
\| \gamma (v) \|_{\mathrm{L}^{2}\left( \partial \Omega \right)} & \\
& + \lambda_{0}  \| u \|_{\mathrm{L}^{2}\left( \Omega \right)} \| v \|_{\mathrm{L}^{2}\left( \Omega \right)} & \\
\ \\
& \leq \left( \ d^{2} \ \| A \|_{\mathrm{L}^{\infty}(\Omega)^{d \times d}}  +   \| B \|_{\mathrm{L}^{2} \to \mathrm{L}^{2}} \ 
\| \gamma \|_{H^{1} \to \mathrm{L}^{2}}^{2} \ \right) 
\| u \|_{H^{1}\left( \Omega \right)} \| v \|_{H^{1}\left( \Omega \right)} 
+ \lambda_{0}  \| u \|_{\mathrm{L}^{2}\left( \Omega \right)} \| v \|_{\mathrm{L}^{2}\left( \Omega \right)} & \\
\ \\
& \leq C \| u \|_{\tilde{a}(B)} \| v \|_{\tilde{a}(B)} & 
\end{align*}
for a sufficiently large constant $C>0$ independent of $u$ and $v$.
\item Now, let us show the closedness of $\tilde{a}(B)$. Therefore let $(u_{n}) \subset D \left( \tilde{a}(B) \right)$ be a  
Cauchy sequence with respect to  $\| \cdot \|_{\tilde{a}(B)}$. From $( \ref{H1_inequ} )$
we conclude that $(u_{n})$ is Cauchy in $H^{1}(\Omega)$. Hence there exists a limit
$u \in H^{1}(\Omega) = D \left( \tilde{a}(B) \right)$ of $(u_{n})$ with respect to $\| \cdot \|_{H^{1}(\Omega)}$. Using
the inequality
\begin{align*}
& \left| \tilde{a}(B)(u_{n}-u, u_{n}-u) \right| \leq & \\
& \left( \ d^{2} \ \| A \|_{\mathrm{L}^{\infty}(\Omega)^{d \times d}}  +   \| B \|_{\mathrm{L}^{2} \to \mathrm{L}^{2}} \ 
\| \gamma \|_{H^{1} \to \mathrm{L}^{2}}^{2} \ \right)  
\| u_{n}-u \|_{H^{1}\left( \Omega \right)}^{2} 
+ \lambda_{0}  \| u_{n}-u, \|_{\mathrm{L}^{2}\left( \Omega \right)}^{2} &
\end{align*}
we saw earlier in the proof of Lemma \ref{form_properties}, we conclude that $u$ is also a limit of $(u_{n})$ 
with respect to $\| \cdot \|_{\tilde{a}(B)}$. 
\end{enumerate}
\end{proof}

\begin{cor}\label{C^1_core}
The subspace $C^{1} \left( \overline{\Omega} \right)$ is dense in $D \left( \tilde{a}(B) \right)$ 
with respect to $\| \cdot \|_{\tilde{a}(B)}$. 
\end{cor}

\begin{proof}
The claim is also implied by the inequality
$$
\left| \tilde{a}(B)(u,v) \right|  \leq 
C \| u \|_{H^{1}\left( \Omega \right)} \| v \|_{H^{1}\left( \Omega \right)} 
+ \lambda_{0}  \| u \|_{\mathrm{L}^{2}\left( \Omega \right)} \| v \|_{\mathrm{L}^{2}\left( \Omega \right)}
$$
for a costant $C > 0$ from the proof of Lemma \ref{form_properties}.\\
\end{proof}

\begin{dfn}
We define the associated operator to $\tilde{a}(B)$ by
\begin{align*}
& D ( \tilde{L}(B) ) = & \\
& \left\{ u \in D \left( \tilde{a}(B) \right) \ | \ \exists v \in \mathrm{L}^{2}\left( \Omega \right) : 
\tilde{a}(B) (u, \varphi) = \langle v, \varphi \rangle_{\mathrm{L}^{2}\left( \Omega \right)} \ \text{ for every } 
\varphi \in D \left( \tilde{a}(B) \right) \right\} & 
\end{align*}
as the domain of 
$\tilde{L}(B) \colon D ( \tilde{L}(B) ) \subset \mathrm{L}^{2}\left( \Omega \right) \to \mathrm{L}^{2}\left( \Omega \right)$ 
which is characterized by 
$$
\tilde{a}(B)(u, \varphi) =  \langle \tilde{L}(B)u, \varphi \rangle_{\mathrm{L}^{2}\left( \Omega \right)}
$$
for $u \in D ( \tilde{L}(B) )$ and every $\varphi \in D \left( \tilde{a}(B) \right)$. 
\end{dfn}

Using Theorem \ref{generator_semigroup} we identify $-\tilde{L}(B)$ as a generator of a $C_{0}$-semigroup 
$( \mathrm{e}^{-t\tilde{L}(B)} )_{t \geq 0}$ 
of contractions in $\mathrm{L}^{2}\left( \Omega \right)$. Let us focus on $a(B)$ and its associated operator $L(B)$ next.

\begin{lem}\label{Semigroup_adjustment}
Define the operator $L(B) = \tilde{L}(B) - \lambda_{0}$ in $\mathrm{L}^{2}\left( \Omega \right)$ on 
$D \left( L(B) \right) = D ( \tilde{L}(B) )$. Then $L(B)$ satisfies the following properties:
\begin{enumerate}[i.)]
\item $L(B)$ is associated to the form $a(B)$, i.e.
$$
a(B)(u, \varphi) =  \langle L(B)u, \varphi \rangle_{\mathrm{L}^{2}\left( \Omega \right)}
$$
holds for every $u \in D ( L(B) )$ and every $\varphi \in D \left( a(B) \right)$.
\item $-L(B)$ generates a $C_{0}$-semigroup 
$\mathrm{e}^{-tL(B)} = \mathrm{e}^{t\lambda_{0}}\mathrm{e}^{-t\tilde{L}(B)}$
in $\mathrm{L}^{2}\left( \Omega \right)$.
\end{enumerate}
\end{lem} 

\begin{proof}
\begin{enumerate}[i.)]
\item The claim is a direct implication of
\begin{align*}
a(B)(u,v) & = \int_{\Omega} A \nabla u \cdot \overline{\nabla v} \d x + 
\int_{\partial \Omega} B \gamma(u) \ \overline{\gamma(v)} \d \sigma(x) \\
& = \tilde{a}(B)(u,v) - \lambda_{0} \int_{\Omega} u \ \overline{v} \d x.
\end{align*}
\item Note that $\mathrm{e}^{t\lambda_{0}}\mathrm{e}^{-t\tilde{L}(B)}$ indeed is a strongly continuous
operator semigroup in $\mathrm{L}^{2}\left( \Omega \right)$. For $u \in D \left( L(B) \right)$ we conclude to
$$
\frac{\partial}{\partial t} \mathrm{e}^{-tL(B)} u = \frac{\partial}{\partial t} \mathrm{e}^{t\lambda_{0}}\mathrm{e}^{-t\tilde{L}(B)} u 
= \lambda_{0} \mathrm{e}^{t\lambda_{0}}\mathrm{e}^{-t\tilde{L}(B)} u - 
\tilde{L}(B) \mathrm{e}^{t\lambda_{0}}\mathrm{e}^{-t\tilde{L}(B)} u  = - L(B) \mathrm{e}^{-tL(B)} u.
$$ 
\end{enumerate}
\end{proof}

\begin{rem} \
\begin{enumerate}[i.)]
\item
By Lemma \ref{generator_adjoint_semigroup} the adjoint operator $L(B)^{\ast}$ of 
$L(B)$ in $\mathrm{L}^{2}\left( \Omega \right)$ is associated to the adjoint form $a(B)^{\ast}$ given by
\begin{equation}
a(B)^{\ast}(u,v) = \int_{\Omega} A^{T} \nabla u \cdot \overline{\nabla v} \d x + 
\int_{\partial \Omega} B^{\ast} \gamma(u) \ \overline{\gamma(v)} \d \sigma(x)
\end{equation} 
for $u, v \in D \left( a(B)^{\ast} \right) = D \left( a(B) \right)$. Please compare with Appendix \ref{results_by_Ouhabaz} and
Proposition 2.2 (iv) on page 6 in \cite{Glueck_Mui_2026}. 
\item Furthermore Lemma \ref{generator_adjoint_semigroup} implies $-L(B)^{\ast}$ to be the generator of the 
adjoint $C_{0}$-semigroup $\mathrm{e}^{-tL(B)^{\ast}}$ in $\mathrm{L}^{2}\left( \Omega \right)$ with
$\mathrm{e}^{-tL(B)^{\ast}} = \left( \mathrm{e}^{-tL(B)} \right)^{\ast}$. 
\end{enumerate}
\end{rem}

\subsection{Proving $\mathrm{e}^{-tL(B)^{\ast}} \in \mathcal{L} \left( \mathrm{L}^{1}\left( \Omega \right) \right)$}

Remember that the main result of this article is to present conditions of the boundary operator 
$B \in \mathrm{L}^{2}\left( \partial \Omega \right)$
such that the solution operator semigroup $\mathrm{e}^{-tL(B)}$ is ultracontractive in $\mathrm{L}^{2}\left( \Omega \right)$, i.e.
$$
\exists \ C, \mu > 0 \ \forall u \in  \mathrm{L}^{2}\left( \Omega \right) \ : \ 
\left\| \mathrm{e}^{-tL(B)} u \right\|_{ \mathrm{L}^{\infty}\left( \Omega \right)}
\leq C t^{-\frac{\mu}{4}} \left\| u \right\|_{ \mathrm{L}^{2}\left( \Omega \right)} 
$$
is true for any time $t \in (0,1]$. We use Nash's inequality (\ref{Nash_inequality}) for our argumentation. 
We will see that it is crucial 
to show that $\mathrm{e}^{-tL(B)^{\ast}}$ are continuous operators in $\mathrm{L}^{1}\left( \Omega \right)$. 
The next proposition states that $\mathrm{e}^{-tL(B)}$ being continuous operators in 
$\mathrm{L}^{\infty}\left( \Omega \right)$ is sufficient.\\

\begin{prp}\label{L1_Linfty_semigroup}
Let $\mathrm{e}^{-tL(B)}$ be a continuous operator in $\mathrm{L}^{\infty}\left( \Omega \right)$ at any time $t \geq 0$. 
Then  $\mathrm{e}^{-tL(B)^{\ast}}$ are continuous operators in $\mathrm{L}^{1}\left( \Omega \right)$ such that
$$
\left\| \mathrm{e}^{-tL(B)^{\ast}} \right\|_{\mathrm{L}^{1} \to \mathrm{L}^{1}} \ \leq \ 
\left\| \mathrm{e}^{-tL(B)} \right\|_{ \mathrm{L}^{\infty} \to \mathrm{L}^{\infty} } 
$$
is true for every time $t \geq 0$.
\end{prp}

\begin{proof}
Let $u \in \mathrm{L}^{2}\left( \Omega \right)$ and note that 
$\mathrm{L}^{2}\left( \Omega \right) \subset \mathrm{L}^{1}\left( \Omega \right)$ is true since 
 $\Omega$ is a bounded Lipschitz domain in $\R^{d}$.
Then $\mathrm{e}^{-tL(B)^{\ast}}u \in \mathrm{L}^{2}\left( \Omega \right)$ is contained
in $\mathrm{L}^{1}\left( \Omega \right)$. Using the Hahn-Banach extension 
theorem there exists a 
$v \in \mathrm{L}^{\infty}( \Omega ) = \left( \mathrm{L}^{1}( \Omega ) \right)^{\ast}$ with 
$\| v \|_{\mathrm{L}^{\infty}(\Omega)} = 1$ such that
$$
\langle \mathrm{e}^{-tL(B)^{\ast}}u, v \rangle \ = \ \left\| \mathrm{e}^{-tL(B)^{\ast}}u \right\|_{\mathrm{L}^{1}(\Omega)}
$$
is true. Note that $v \in \mathrm{L}^{2} \left( \Omega \right)$ is implied by 
$\mathrm{L}^{\infty} \left( \Omega \right) \subset \mathrm{L}^{2} \left( \Omega \right)$ since $\Omega$ is bounded.
Hence we conclude to
\begin{align*}
\left\| \mathrm{e}^{-tL(B)^{\ast}}u \right\|_{\mathrm{L}^{1}(\Omega)} & = 
\left|  \langle \mathrm{e}^{-tL(B)^{\ast}}u, v \rangle \right| = \left|  \langle u, \mathrm{e}^{-tL(B)}v \rangle \right| \\
\ \\
& \leq \left\| \mathrm{e}^{-tL(B)}v \right\|_{\mathrm{L}^{\infty}(\Omega)} \ \| u \|_{\mathrm{L}^{1}(\Omega)}
\leq \left\| \mathrm{e}^{-tL(B)} \right\|_{ \mathrm{L}^{\infty} \to \mathrm{L}^{\infty} } \
\left\| v \right\|_{\mathrm{L}^{\infty}(\Omega)} \ \| u \|_{\mathrm{L}^{1}(\Omega)} \\
\ \\
& = \left\| \mathrm{e}^{-tL(B)} \right\|_{ \mathrm{L}^{\infty} \to \mathrm{L}^{\infty} } \ \| u \|_{\mathrm{L}^{1}(\Omega)} 
\end{align*}
is true where we used that $\mathrm{e}^{-tL_{B}}$ is continuous in $\mathrm{L}^{\infty}\left( \Omega \right)$
and $ \mathrm{e}^{-tL(B)^{\ast}}$ is the adjoint operator of $\mathrm{e}^{-tL_{B}}$ in $\mathrm{L}^{2}\left( \Omega \right)$.
Therefore $\mathrm{e}^{-tL(B)^{\ast}}$ is continuously extendable to a bounded operator in 
$\mathrm{L}^{1}(\Omega)$ with
$\left\| \mathrm{e}^{-tL(B)^{\ast}} \right\|_{\mathrm{L}^{1} \to \mathrm{L}^{1}} \ \leq \ 
\left\| \mathrm{e}^{-tL(B)} \right\|_{ \mathrm{L}^{\infty} \to \mathrm{L}^{\infty} }$
since $\mathrm{L}^{2}(\Omega)$ is dense in $\mathrm{L}^{1}(\Omega)$.\\
\end{proof}

In the following we therefore focus on finding conditions on $B$ such that the solution semigroup operators
$\mathrm{e}^{-tL_{B}}$ are continuous in $\mathrm{L}^{\infty}\left( \Omega \right)$ at any time $t \geq 0$.

\subsubsection{Domination of the auxiliary semigroup $\mathrm{e}^{-t\tilde{L}_{B}}$}

Let the following properties be true:
\begin{enumerate}[i)]
\item
Let $B_{pos} \in \mathcal{L} \left( \mathrm{L}^{2}( \partial \Omega ) \right)$ a positive operator with
$
\left| B u \right| \ \leq \ B_{pos} \left| u \right|
$
for $u \in \mathrm{L}^{2}( \partial \Omega )$. 
\item Let $B_{pos}$ be a continuous operator in $\mathrm{L}^{\infty}( \partial \Omega )$. 
\item Let $h \in \mathrm{L}^{2}( \partial \Omega )$ be a function with $\frac{1}{2} \ \leq \ h \ \leq \ 2$
almost everywhere in $\partial \Omega$. \\
\end{enumerate}
\ \\
Notice that the positivity of $B_{pos}$ and strict positivity and boundedness of $h$ implies
\begin{equation} \label{B_{pos}_inequ}
0 \leq \frac{1}{2} B_{pos} \mathbb{1} \leq B_{pos}h \leq \| B_{pos} \|_{\mathrm{L}^{\infty} \to \mathrm{L}^{\infty}} 2 =
4 \| B_{pos} \|_{\mathrm{L}^{\infty} \to \mathrm{L}^{\infty}} \frac{1}{2} \leq 
4 \| B_{pos} \|_{\mathrm{L}^{\infty} \to \mathrm{L}^{\infty}} h
\end{equation}
pointwise almost everywhere in $\partial \Omega$. We state the following definition.

\begin{dfn}
We define $\beta = 4 \| B_{pos} \|_{\mathrm{L}^{\infty} \to \mathrm{L}^{\infty}}$ and
$$
\overline{B} u \ = \ - B_{pos}u - \left( \beta h - B_{pos} h \right)
$$
for every $u \in \mathrm{L}^{2}( \partial \Omega )$.
\end{dfn}

\begin{rem} \
\begin{enumerate}[i)]
\item The general idea is to introduce an auxiliary boundary operator $\overline{B}$ to prove the inequality
$$
\left| \mathrm{e}^{-tL(B)}u \right| \ \leq \ \mathrm{e}^{-tL(\overline{B})} \left| u \right|
$$
pointwise almost everywhere in $\Omega$ for every $u \in \mathrm{L}^{2}( \Omega )$. Then we focus on the auxiliary 
semigroup $\mathrm{e}^{-tL(\overline{B})}$ and a continuity in $\mathrm{L}^{\infty}( \Omega )$.
\item The term $\beta h - B_{pos} h \in  \mathrm{L}^{2}( \partial \Omega )$ in $\overline{B}$ is included for exactly that
purpose.
\item Note that $\beta h - B_{pos} h \geq 0$ is true almost everywhere in $\partial \Omega$ by (\ref{B_{pos}_inequ}).
\end{enumerate}
\end{rem}

\begin{lem}\label{inequ_semigroup_B_overlineB}
The operator semigroup $\mathrm{e}^{-t\tilde{L}(B)}$ to $B$ is dominated by the operator semigroup
 $\mathrm{e}^{-t\tilde{L}(\overline{B})}$ to $\overline{B}$, i.e.
$$
\left| \mathrm{e}^{-t\tilde{L}(B)}u \right| \ \leq \ \mathrm{e}^{-t\tilde{L}(\overline{B})} \left| u \right|
$$ 
holds pointwise almost everywhere in $\Omega$ for every $u \in   \mathrm{L}^{2}( \Omega )$ where $\lambda_{0}$ 
in the auxiliary forms $\tilde{a}(B)$ and $\tilde{a}(\overline{B})$ is chosen sufficiently large that both forms are accretive. 
Furthermore 
$$
\left| \mathrm{e}^{-tL(B)}u \right| \ \leq \ \mathrm{e}^{-tL(\overline{B})} \left| u \right|
$$
is a direct implication.
\end{lem}

\begin{proof}
We use Theorem 2.21 on page 60 in \cite{Ouhabaz05} for our argumentation. Therefore we have to use the accretive forms
$\tilde{a}(B)$ and $\tilde{a}(\overline{B})$. We chose $\lambda_{0}$ sufficiently large that both forms are accretive.
\begin{enumerate}[i)]
\item First we have to show that $\mathrm{e}^{-t\tilde{L}(\overline{B})}$ are all positive operators. Let 
$u \in C^{1}( \overline{\Omega} )$.  Then $\partial_{k}(\Re u)^{+} \partial_{l}  (\Re u)^{-} = 0$ is true almost everywhere
in $\Omega$ due to disjoint supports where we remember that $(\Re u)^{+},  (\Re u)^{-} \in H^{1}(\Omega)$.
So
$$
\langle A\nabla (\Re u)^{+}, \nabla (\Re u)^{-} \rangle_{\mathrm{L}^{2}\left( \Omega \right)^{d}} = 0 
$$
is implied. Hence we infer
\begin{align*} 
& \tilde{a}(\overline{B}) \left(  (\Re u)^{+},  (\Re u)^{-} \right) = \int_{\partial \Omega} \overline{B} (\Re u)^{+} \ (\Re u)^{-} \d \sigma(x)
& \\
& = - \int_{\partial \Omega} \underbrace{B_{pos} (\Re u)^{+}}_{\geq 0} \ \underbrace{(\Re u)^{-}}_{\geq 0} \d \sigma(x) 
- \int_{\partial \Omega} \underbrace{\left( \beta h - B_{pos} h \right)}_{\geq 0} \ \underbrace{(\Re u)^{-}}_{\geq 0} \d \sigma(x) 
 \leq 0. &
\end{align*}
Using Theorem 2.6 on page 50 in \cite{Ouhabaz05} we conclude that $\mathrm{e}^{-t\tilde{L}(\overline{B})}$ are positive operators
since $C^{1}( \overline{\Omega} )$ is a core of $\tilde{a}(\overline{B})$. 
\item First note that $H^{1}(\Omega)$ is an ideal in itself by Proposition 2.20 on page 59 in \cite{Ouhabaz05} since
$\mathrm{e}^{-t\tilde{L}(\overline{B})}$ are positive. Now let $u,v \in H^{1}(\Omega)$ be such that $u\overline{v} \geq 0$. 
In the following we argue that
$$
\tilde{a}(\overline{B})(|u|, |v|) \ \leq \ \Re \tilde{a}(B)(u, v)
$$
is implied. First we conclude that
\begin{align*}
- \Re \langle B\gamma(u), \gamma(v) \rangle_{\mathrm{L}^{2}( \partial \Omega )} & 
\leq \left| \langle B\gamma(u), \gamma(v) \rangle_{\mathrm{L}^{2}( \partial \Omega )} \right| 
\leq \langle \left| B\gamma(u) \right|, |\gamma(v)| \rangle_{\mathrm{L}^{2}( \partial \Omega )} & \\
\ \\
& \leq \langle B_{pos} \left| \gamma(u) \right|, |\gamma(v)| \rangle_{\mathrm{L}^{2}( \partial \Omega )} & \\
\ \\
& \leq \langle B_{pos} \left| \gamma(u) \right| + \left( \beta h - B_{pos} h \right), |\gamma(v)| \rangle_{\mathrm{L}^{2}( \partial \Omega )}
& \\
\ \\
& =  \langle - \overline{B} \left| \gamma(u) \right|, \left| v \right|  \rangle_{\mathrm{L}^{2}( \partial \Omega )} & 
\end{align*}
is true. Hence
$\langle \overline{B} \left| \gamma(u) \right|, \left| v \right|  \rangle_{\mathrm{L}^{2}( \partial \Omega )} 
\leq \Re \langle B\gamma(u), \gamma(v) \rangle_{\mathrm{L}^{2}( \partial \Omega )}$
is implied. \\
We need similar results for the remaining part of the forms. Note that $\mathrm{e}^{-t\tilde{L}(0)}$ are positive operators. Hence
$$
\left| \mathrm{e}^{-t\tilde{L}(0)}w \right| \leq \mathrm{e}^{-t\tilde{L}(0)} \left| w \right|
$$
is true for every $w \in \mathrm{L}^{2}(\Omega)$. By Theorem 2.21 on page 60 in \cite{Ouhabaz05} 
$$
\tilde{a}(0)(|u|, |v|) \ \leq \ \Re \tilde{a}(0)(u, v)
$$
is implied. Hence we conclude to 
$$
\tilde{a}(\overline{B})(|u|, |v|) \ \leq \ \Re \tilde{a}(B)(u, v)
$$
which proves the claim by using Theorem 2.21 in \cite{Ouhabaz05}  once again.
\end{enumerate}
\end{proof}

\subsubsection{Proving $\mathrm{e}^{-tL(B)} \in \mathcal{L} \left( \mathrm{L}^{\infty}\left( \Omega \right) \right)$}

Again let the following properties be true:
\begin{enumerate}[i)]
\item
Let $B_{pos} \in \mathcal{L} \left( \mathrm{L}^{2}( \partial \Omega ) \right)$ a positive operator with
$
\left| B u \right| \ \leq \ B_{pos} \left| u \right|
$
for $u \in \mathrm{L}^{2}( \partial \Omega )$. 
\item Let $B_{pos}$ be a continuous operator in $\mathrm{L}^{\infty}( \partial \Omega )$. 
\end{enumerate}
\ \\
Let us prove the existence of such an auxiliary function $h$ in the previous section that implies 
$$
\left| \mathrm{e}^{-tL(B)}u \right| \ \leq \ \mathrm{e}^{-tL(\overline{B})} \left| u \right|
$$
almost everywhere in $\Omega$ for any $u \in \mathrm{L}^{2}( \Omega )$. Remember that 
$\beta = 4 \| B_{pos} \|_{\mathrm{L}^{\infty} \to \mathrm{L}^{\infty}}$. We define the form $a(-\beta)$ by
$$
a(-\beta)(u,v) = \int_{\Omega} A \nabla u \cdot \overline{\nabla v} \d x - 
\int_{\partial \Omega} \beta \gamma(u) \ \overline{\gamma(v)} \d \sigma(x)
$$
for $u,v \in D\left( a(-\beta) \right) = H^{1}(\Omega)$ where $\beta \gamma(u)$ is the product with
the constant $\beta$. 

\begin{lem}\label{definition_h}
There exists a function $h \in D \left( L( \overline{B}) \right) \cap D \left( L(-\beta) \right)$ with 
\begin{enumerate}[i)]
\item $\frac{1}{2} \ \leq \ h \ \leq \ 2$ almost everywhere in $\Omega$ and
\item $L(\overline{B})u = L(-\beta)u$ in $\mathrm{L}^{2}( \Omega )$.
\end{enumerate}
\end{lem}

\begin{proof}
\begin{enumerate}[i)]
\item We use Theorem 4.3 on page 13 in \cite{Nittka2011} and conclude that the restriction of 
$\mathrm{e}^{-tL(-\beta)}$ to $C(\overline{\Omega})$ is a strongly-continuous operator semigroup
in $C(\overline{\Omega})$. Furthermore $-L(-\beta)$ restricted to $C(\overline{\Omega})$ is its generator.
\item Using Lemma 2.10 on page 18 in \cite{ACSVV15} we conclude that
$$
\left\| \mu ( \mu + L(-\beta) )^{-1} \mathbb{1} - \mathbb{1} \right\|_{\infty} \ \to \ 0 
$$
holds for $\mu \to \infty$ where $\mathbb{1} \in C(\overline{\Omega})$ is a constant function mapping 
$\overline{\Omega}$ to $1$.
\item Choose $\mu_{0} > 0$ sufficiently large such that 
$$
h = \mu_{0} ( \mu_{0} + L(-\beta) )^{-1} \mathbb{1} \ \in \ C(\overline{\Omega}) \cap D \left( L(-\beta) \right)
$$
satsifies $\frac{1}{2} \ \leq \ h \ \leq \ 2$ everywhere in $\overline{\Omega}$.
\item For every $v \in H^{1}(\Omega)$ we argue that
\begin{align*}
\langle L(-\beta)h, v \rangle_{\mathrm{L}^{2}(\Omega)} & =  a(-\beta)(h,v)   & \\
& = \int_{\Omega} A \nabla h \cdot \overline{\nabla v} \d x - 
\int_{\partial \Omega} \beta h \ \overline{\gamma(v)} \d \sigma(x) & \\
&  = \int_{\Omega} A \nabla h \cdot \overline{\nabla v} \d x + 
\int_{\partial \Omega} \overline{B} h \ \overline{\gamma(v)} \d \sigma(x) & \\
& = a( \overline{B} )(h,v) & 
\end{align*}
is true. Hence $h \in D \left( L( \overline{B}) \right)$ with $L( \overline{B})h = L( -\beta)h$.
\end{enumerate}
\end{proof}

\begin{lem}\label{Linfty_overlineB}
Let $h$ and $\mu_{0} \geq 0$ be as described in Lemma \ref{definition_h}. For any time $t \geq 0$
the auxiliary semigroup operators $\mathrm{e}^{-tL(\overline{B})}$ are continuous in $\mathrm{L}^{\infty}( \Omega )$
with
$$
\left\| \mathrm{e}^{-tL(\overline{B})} \right\|_{\mathrm{L}^{\infty} \to \mathrm{L}^{\infty}} \ \leq \ 4 \mathrm{e}^{\mu_{0}t}.
$$
\end{lem}

\begin{proof}
\begin{enumerate}[i)]
\item Define $w = \left( \mu_{0} + L(\overline{B}) \right)h \in \mathrm{L}^{2}(\Omega)$. Then
$w = \left( \mu_{0} + L(-\beta) \right)h = \mu_{0} \geq 0$ is true almost everywhere in $\Omega$. Furthermore
$$
h = \left( \mu_{0} + L(-\beta) \right)^{-1}w = \int_{0}^{\infty} \mathrm{e}^{-\mu_{0}s} \mathrm{e}^{-sL(\overline{B})}w \d s
$$
is implied. Then we use the continuity of $\mathrm{e}^{-tL(\overline{B})}$ in $\mathrm{L}^{2}(\Omega)$ as well
as its positivity to conclude to
\begin{align*}
\mathrm{e}^{-tL(\overline{B})} h & = \mathrm{e}^{-tL(\overline{B})} 
\int_{0}^{\infty} \mathrm{e}^{-\mu_{0}s} \mathrm{e}^{-sL(\overline{B})}w \d s & \\
& = \int_{0}^{\infty} \mathrm{e}^{-\mu_{0}s} \mathrm{e}^{-(s+t)L(\overline{B})}w \d s & \\
& = \mathrm{e}^{\mu_{0}t}  \int_{t}^{\infty} \mathrm{e}^{-\mu_{0}\tau} \mathrm{e}^{-\tau L(\overline{B})}w \d \tau & \\
& \leq \mathrm{e}^{\mu_{0}t}  \int_{0}^{\infty} \mathrm{e}^{-\mu_{0}\tau} \mathrm{e}^{-\tau L(\overline{B})}w \d \tau & \\
& = \mathrm{e}^{\mu_{0}t} \left( \mu_{0} + L(-\beta) \right)^{-1}w =  \mathrm{e}^{\mu_{0}t} h \leq 2 \mathrm{e}^{\mu_{0}t} & \\
\end{align*}
almost everywhere in $\Omega$ where 
$$
\int_{t}^{\infty} \mathrm{e}^{-\mu_{0}\tau} \mathrm{e}^{-\tau L(\overline{B})}w \d \tau \leq
\int_{0}^{\infty} \mathrm{e}^{-\mu_{0}\tau} \mathrm{e}^{-\tau L(\overline{B})}w \d \tau
$$ 
is true since $\mathrm{e}^{-\tau L(\overline{B})}w \geq 0$ is implied almost everywhere in $\Omega$ since  
$\mathrm{e}^{-tL(\overline{B})}$ are positive operators for any time $t \geq 0$ and $w \geq 0$ is a positive function.
\item Now let $u \in \mathrm{L}^{\infty}(\Omega)$ be arbitrary but fixed. Then
\begin{align*}
\left| \mathrm{e}^{-tL(\overline{B})} u \right| & \leq  \mathrm{e}^{-tL(\overline{B})} |u| \leq \| u \|_{\infty}
\mathrm{e}^{-tL(\overline{B})} \mathbb{1} = 2 \| u \|_{\infty} \mathrm{e}^{-tL(\overline{B})} \frac{1}{2} & \\
& \leq 2 \| u \|_{\infty} \mathrm{e}^{-tL(\overline{B})} h \leq  4 \mathrm{e}^{\mu_{0}t} \|u\|_{\infty} & 
\end{align*}
is implied by using the positivty of $\mathrm{e}^{-tL(\overline{B})}$ once again.
\end{enumerate}
\end{proof}

We use a combination of the Lemmata \ref{Linfty_overlineB} and \ref{inequ_semigroup_B_overlineB} as well as
Proposition \ref{L1_Linfty_semigroup} to formulate the following simple corollary.

\begin{cor}\label{L1_Nash}
Suppose there exists a postive operator $B_{pos} \in \mathcal{L} \left( \mathrm{L}^{2}( \partial \Omega ) \right)$ that 
satisfies
$$
\left| B u \right| \ \leq \ B_{pos} \left| u \right|
$$
for every $u \in \mathrm{L}^{2}( \partial \Omega )$. Furthermore let $h$ and $\mu_{0} \geq 0$ be as described in 
Lemma \ref{definition_h}. For any time $t \geq 0$ the
solution semigroup operators $\mathrm{e}^{-tL(B)^{\ast}}$ are continuous in $\mathrm{L}^{1}\left( \Omega \right)$ 
such that
$$
\left\| \mathrm{e}^{-tL(B)^{\ast}} \right\|_{\mathrm{L}^{1} \to \mathrm{L}^{1} } \ \leq \ 4 \mathrm{e}^{\mu_{0}t}
$$
\end{cor}

\section{Main theorem}

Let $\Omega \subset \R^{d}$ with $d \in \N$ be a bounded Lipschitz domain and $A \colon \Omega \to \R^{d \times d}$ be
a matrix-valued function such that every coefficient $a_{ij} \colon \Omega \to \R$ is a bounded and measurable function 
on $\Omega$. Furthermore let $A$ be uniformly elliptic on $\Omega$, i.e.
$$
\exists \alpha > 0 \ \forall \xi \in \C^{d}, x \in \Omega \ : \ 
\Re \left(  \left( A(x) \xi \right)^{T} \overline{\xi} \right) \geq \alpha \left| \xi \right|^{2}. 
$$  
Furhermore let (\ref{Robin_Condition}) be characterized by an operator 
$B \in \mathcal{L}\left( \mathrm{L}^{2}\left( \partial \Omega \right) \right)$.
Suppose there exists a postive operator $B_{pos} \in \mathcal{L} \left( \mathrm{L}^{2}( \partial \Omega ) \right)$ that 
satisfies
$$
\left| B u \right| \ \leq \ B_{pos} \left| u \right|
$$
for every $u \in \mathrm{L}^{2}( \partial \Omega )$.

\begin{thm}\label{Main_result}
Under these conditions of $\Omega$, $A$ and $B$ the solution semigroup operators
$\mathrm{e}^{-tL(B)}$ is ultracontractive in $\mathrm{L}^{2}\left( \Omega \right)$ at any time $t > 0$. 
In particular there are constants $C>0$ and $\mu_{0} > 0$ satisfying
$$
\left\| \mathrm{e}^{-tL(B)} u \right\|_{\mathrm{L}^{\infty} \left( \Omega \right) } \leq C t^{-\frac{d}{4}} \mathrm{e}^{t \mu_{0}}
\left\| u \right\|_{\mathrm{L}^{2} \left( \Omega \right) }
$$
for every time $t > 0$ and any $u \in \mathrm{L}^{2}\left( \Omega \right)$. 
\end{thm}

\begin{rem}
Please compare with Theorem \ref{Theorem_Jochen} from \cite{Glueck_Mui_2026} and notice that Theorem \ref{Main_result}  
offers the same result:
$$
\exists C > 0 \ \forall u \in \mathrm{L}^{2}\left( \Omega \right) \ : \ 
\left\| \mathrm{e}^{-tL_{B}}u \right\|_{\mathrm{L}^{\infty}\left( \Omega \right)} \ \leq \ C t^{-\frac{d}{4}} 
\left\| u \right\|_{\mathrm{L}^{2}\left( \Omega \right)}
$$
for every time $t \in (0,1]$.
\end{rem}

\section{Proof of Theorem \ref{Main_result}}

In the following always let $\Omega$, $A$ and $B$ be as demanded in Theorem \ref{Main_result}.

\subsection{Proving 
$\mathrm{e}^{-tL(B)^{\ast}} \in \mathcal{L} \left( \mathrm{L}^{1}\left( \Omega \right), \mathrm{L}^{2}\left( \Omega \right) \right)$}

We show that the adjoint operators $\mathrm{e}^{-tL(B)^{\ast}} \in \mathcal{L} \left( \mathrm{L}^{2}\left( \Omega \right) \right)$ 
are also continuous operators from $\mathrm{L}^{1}\left( \Omega \right)$ to $\mathrm{L}^{2}\left( \Omega \right)$ at every time 
$t > 0$. The main argument is Nash's inequality (\ref{Nash_inequality}) and 
$$
\Re \tilde{a}(B)^{\ast}(v,v) = \Re \tilde{a}(B)(v,v) \geq  \frac{\alpha}{2} \| \nabla v \|_{\mathrm{L}^{2}(\Omega)^{d}}^{2}
$$
for every $v \in D \left( \tilde{a}(B) \right)$ in Remark (\ref{a(B)_Nash}). Furthermore  
$\mathrm{e}^{-tL(B)^{\ast}}$ are continuous operators in $\mathrm{L}^{1}\left( \Omega \right)$ at every time $t>0$ by
Corollary \ref{L1_Nash} which is crucial to our argumentation. Setting
$v(t) = \mathrm{e}^{-tL(B)^{\ast}}u$ for $u \in D(L(B)^{\ast})$ we infer that 
$v(t) \in \mathrm{L}^{2}\left( \Omega \right)$ is also contained in $\mathrm{L}^{1}\left( \Omega \right)$ since $\Omega$
is bounded. Furthermore $v(t)$ is an element of $D\left( \tilde{a}(B) \right) = H^{1}(\Omega)$.
Therefore Nash's inequality beautifully connects all of the involved function spaces.\\

\begin{lem} \label{adjoint_semigroup_L1_L2}
For any $t > 0$ the operators $\mathrm{e}^{-tL(B)^{\ast}}$ are continuous from 
$\mathrm{L}^{1}\left( \Omega \right)$ to $\mathrm{L}^{2}\left( \Omega \right)$. In particular there exists
a constant $C > 0$ such that
$$
\| \mathrm{e}^{-tL(B)^{\ast}} u \|_{\mathrm{L}^{2}\left( \Omega \right) } \ \leq \ 
C t^{-d/4} \mathrm{e}^{-t\mu_{0}} \ \left\| u \right\|_{\mathrm{L}^{1} \left( \Omega \right)}
$$
is true for every $u \in \mathrm{L}^{1} \left( \Omega \right)$ and $t>0$.
\end{lem}

\begin{proof}
\begin{enumerate}[i)]
\item Please notice that we have focus on the auxiliary form $\tilde{a}(B)^{\ast}$ from Section \ref{Preparations} 
to be able to use 
$$
\Re \tilde{a}(B)^{\ast}(v,v)  \geq  \frac{\alpha}{2} \| \nabla v \|_{\mathrm{L}^{2}(\Omega)^{d}}^{2}
$$
as described above. Therefore we include $\lambda_{0}$ sufficiently large such that
the auxiliary form $\tilde{a}(B)^{\ast}$ defined as 
$$
\tilde{a}(B)^{\ast}(u,v) \ = \ \int_{\Omega} A^{T} \nabla u \cdot \overline{\nabla v} \d x + 
\int_{\partial \Omega} B^{\ast} \gamma(u) \ \overline{\gamma(v)} \d \sigma(x) + \lambda_{0} \int_{\Omega} u\overline{v} \d x
$$
for $u,v \in D \left( \tilde{a}(B)^{\ast} \right) = D \left( \tilde{a}(B) \right) = H^{1} \left( \Omega \right)$ is accretive.
We argue that 
$$
\Re \tilde{a}(B)^{\ast}(u,u) = \Re \tilde{a}(B)(u,u)
$$ 
is true. Hence we conclude to
\begin{equation} \label{Ellipticity_A}
\Re \tilde{a}(B)^{\ast}(u,u) \ = \ \Re \tilde{a}(B)(u,u) \ \geq \ \frac{\alpha}{2} \| \nabla u \|_{\mathrm{L}^{2}(\Omega)^{d}}^{2}
\end{equation}
for any $u \in D \left( \tilde{a}(B)^{\ast} \right)$ by Remark (\ref{a(B)_Nash}).
\item Now we can use Nash's inequality. Therefore let $\mathrm{e}^{-t\tilde{L}(B)^{\ast}}$ be the 
associated semigroup operators to the auxiliary form $\tilde{a}(B)^{\ast}$. 
Furthermore let $u \in D ( \tilde{L}(B)^{\ast} )$ and $t > 0$ be arbitrary but fixed. We define
$$
v(t) = \mathrm{e}^{-t\lambda_{0}} \mathrm{e}^{-tL(B)^{\ast}} u = \mathrm{e}^{-t\tilde{L}(B)^{\ast}} u \in D ( \tilde{L}(B)^{\ast} ) \subset \mathrm{L}^{2} \left( \Omega \right).
$$
Furthermore $v(t)$ and $u$ are both contained
in $\mathrm{L}^{1} \left( \Omega \right)$ since $\Omega$ is bounded. Hence 
$$
\left\| v(t) \right\|_{\mathrm{L}^{1} \left( \Omega \right)} \ \leq \ \mathrm{e}^{- t \lambda_{0}} \ 
\left\| \mathrm{e}^{-tL(B)^{\ast}} \right\|_{\mathrm{L}^{1} \to \mathrm{L}^{1}} \ 
\left\| u \right\|_{\mathrm{L}^{1} \left( \Omega \right)}
\leq \ 4 \mathrm{e}^{t(\mu_{0} - \lambda_{0})} \left\| u \right\|_{\mathrm{L}^{1} \left( \Omega \right)}
$$
is implied by Corollary \ref{L1_Nash}. We start the argumentation with
\begin{align*}
\frac{\partial}{\partial t} \| v(t) \|^{2}_{\mathrm{L}^{2}\left( \Omega \right)} & = 
2 \Re \langle v^{\prime}(t), v(t) \rangle_{\mathrm{L}^{2}\left( \Omega \right)}
 = - 2 \Re \langle \tilde{L}(B)^{\ast} v(t), v(t)  \rangle_{\mathrm{L}^{2}\left( \Omega \right)} & \\
& = -2 \Re  \tilde{a}(B)^{\ast}(v(t), v(t))
\leq - \alpha  \| \nabla u \|_{\mathrm{L}^{2}(\Omega)^{d}}^{2} &
\end{align*}
by (\ref{Ellipticity_A}). Now we use Nash's inequality (\ref{Nash_inequality}) to infer that
$$
\frac{\partial}{\partial t} \| v(t) \|^{2}_{\mathrm{L}^{2}\left( \Omega \right)} \ \leq \
 - \alpha  \| \nabla u \|_{\mathrm{L}^{2}(\Omega)^{d}}^{2} \ \leq \ - \frac{\alpha}{C_{d}}  
\frac{ \| v(t) \|_{\mathrm{L}^{2}(\Omega)}^{2+\frac{4}{d}} }{ \| v(t) \|_{\mathrm{L}^{1}(\Omega)}^{\frac{4}{d}}  }
$$
is true. Therefore we use the chain rule to further argue that 
\begin{align*}
\frac{\partial}{\partial t} \left( \| v(t) \|^{2}_{\mathrm{L}^{2}\left( \Omega \right)} \right)^{-\frac{2}{d}} 
\ & \geq \ \left( -\frac{2}{d}  \right) \left( \| v(t) \|^{2}_{\mathrm{L}^{2}\left( \Omega \right)} \right)^{-\frac{2}{d}-1}
 \left( - \frac{\alpha}{C_{d}} \right)  
\frac{ \| v(t) \|_{\mathrm{L}^{2}(\Omega)}^{2+\frac{4}{d}} }{ \| v(t) \|_{\mathrm{L}^{1}(\Omega)}^{\frac{4}{d}}  } \\
\ \\
& = \ \frac{2 \alpha}{d C_{d}} \ \| v(t) \|_{\mathrm{L}^{1}(\Omega)}^{-\frac{4}{d}} & \\
\ \\
& \geq \ \frac{2 \alpha }{d C_{d}} \ 4^{-4/d} \exp \left( -\frac{4t}{d}(\mu_{0} - \lambda_{0}) \right) 
\left\| u \right\|_{\mathrm{L}^{1} \left( \Omega \right)}^{-\frac{4}{d}} &
\end{align*}
is implied. Hence we conclude to
\begin{align*}
\| v(t) \|_{\mathrm{L}^{2}\left( \Omega \right)}^{-\frac{4}{d}} 
& = \ \int_{0}^{t}  \frac{\partial}{\partial t} \ \| v(t) \|_{\mathrm{L}^{2}\left( \Omega \right)}^{-\frac{4}{d}} \d t 
+ \| u \|_{\mathrm{L}^{2}\left( \Omega \right)}^{-\frac{4}{d}} & \\
\ \\
& \geq \  \frac{2 \alpha }{d C_{d}} \ 4^{-4/d} \exp \left( -\frac{4t}{d}(\mu_{0} - \lambda_{0}) \right) t
\left\| u \right\|_{\mathrm{L}^{1} \left( \Omega \right)}^{-\frac{4}{d}} & 
\end{align*} 
and end up with
\begin{equation}
\| \mathrm{e}^{-t\tilde{L}(B)^{\ast}} u \|_{\mathrm{L}^{2}\left( \Omega \right) } \ \leq \
4 \left( \frac{d C_{d}}{2 \alpha} \right)^{d/4} \ t^{-d/4} \ \exp \left( t(\mu_{0} - \lambda_{0}) \right) \
\left\| u \right\|_{\mathrm{L}^{1} \left( \Omega \right)}
\end{equation}
for any function $u \in D ( \tilde{L}(B)^{\ast} )$.
\item Note that $D ( \tilde{L}(B)^{\ast} )$ is a dense subset in $\mathrm{L}^{2} \left( \Omega \right)$ due to 
$\tilde{L}(B)^{\ast}$ being a generator of a $C_{0}$-semigroup in $\mathrm{L}^{2} \left( \Omega \right)$.
Furthermore $\Omega$ is bounded and therefore $D ( \tilde{L}(B)^{\ast} )$ is even a dense subset 
in $\mathrm{L}^{1} \left( \Omega \right)$ with respect to $\| \cdot \|_{\mathrm{L}^{1} \left( \Omega \right)}$. 
Finally we argue that there exists an unique extension of $\mathrm{e}^{-tL_{B}^{\ast}}$ to a continuous operator from 
$\mathrm{L}^{1}\left( \Omega \right)$ to $\mathrm{L}^{2}\left( \Omega \right)$ satisfying
$$
\exists C > 0 \ \forall u \in \mathrm{L}^{1} \left( \Omega \right) \ : \
\| \mathrm{e}^{-tL(B)^{\ast}} u \|_{\mathrm{L}^{2}\left( \Omega \right) } \ \leq \ 
C t^{-d/4} \ \mathrm{e}^{t\mu_{0}} \ \left\| u \right\|_{\mathrm{L}^{1} \left( \Omega \right)}
$$
where we used $\mathrm{e}^{-t\tilde{L}(B)^{\ast}} = \mathrm{e}^{-t\lambda_{0}} \mathrm{e}^{-tL(B)^{\ast}}$.
\end{enumerate}
\end{proof}

\subsection{A duality argumenation}

From Lemma \ref{adjoint_semigroup_L1_L2} we conclude that there exists an adjoint operator 
$T(t) \in \mathcal{L} \left( \mathrm{L}^{2}\left( \Omega \right), \mathrm{L}^{\infty}\left( \Omega \right)  \right)$ to 
$\mathrm{e}^{-tL(B)^{\ast}} \in 
\mathcal{L} \left( \mathrm{L}^{1}\left( \Omega \right), \mathrm{L}^{2}\left( \Omega \right)  \right)$ 
satisfying
\begin{enumerate} [i)]
\item
$\langle \mathrm{e}^{-tL(B)^{\ast}}u, v  \rangle = \langle u, T(t)v \rangle$
for every $u \in \mathrm{L}^{1}\left( \Omega \right)$ and $v \in \mathrm{L}^{2}\left( \Omega \right)$
\item
$\|  T(t) \|_{\mathrm{L}^{2} \to \mathrm{L}^{\infty}} = \|  \mathrm{e}^{-tL(B)^{\ast}} \|_{\mathrm{L}^{1} \to \mathrm{L}^{2}} 
\leq C t^{-d/4} \mathrm{e}^{t\mu_{0}}$.
\end{enumerate}
On the other hand $\mathrm{e}^{-tL(B)^{\ast}}$
is the adjoint operator of $\mathrm{e}^{-tL(B)}$ in $\mathrm{L}^{2} \left( \Omega \right)$ by Lemma 
\ref{generator_adjoint_semigroup}. Furthermore
$\mathrm{L}^{2} \left( \Omega \right)$ is a subset of $\mathrm{L}^{1} \left( \Omega \right)$ since $\Omega$ is bounded
and therefore we conclude
$$
\langle \mathrm{e}^{-tL(B)} u, v \rangle 
= \langle u, \mathrm{e}^{-tL(B)^{\ast}}v \rangle = \langle T(t)u, v \rangle  
$$
is true for every $u,v \in \mathrm{L}^{2} \left( \Omega \right)$. 
Hence $\mathrm{e}^{-tL(B)} u = T(t)u$ is implied
for every $u \in \mathrm{L}^{2} \left( \Omega \right)$ and any $t > 0$ which finally proves Theorem \ref{Main_result}.

\newpage

\appendix

\section{Useful results on forms and $C_{0}$-semigroups} \label{results_by_Ouhabaz}

Let us collect some known facts from \cite{Ouhabaz05} as a guideline to our own argumentation. Let $H$ be
a Hilbert space over $\mathbb{K} = \C$ or $\R$ and $a \colon D(a) \times D(a) \to \mathbb{K}$ be a sesquilinear form.

\begin{dfn} \ 
\begin{enumerate}[i.)]
\item We call $a$ densely defined if $D(a)$ is dense in $H$.
\item We call $a$ accretive if $\Re a(u,u) \geq 0$ holds for every $u \in D(a)$.
\item We call $a$ continuous if $a$ is accretive and there exists a constant $M \geq 0$ such that
$$
\left| a(u,v) \right| \leq M \| u \|_{a} \| v \|_{a}
$$
holds for every $u, v \in D(a)$ where  $\| u \|_{a} = \left( \Re a(u,u) + \| u \| \right)^{\frac{1}{2}}$.
\item We call $a$ a closed form if $\left( D(a), \| \cdot \|_{a} \right)$ is a complete space.
\item The adjoint form $a^{\ast}$ of $a$ is defined by
$$
a^{\ast} (u,v) = \overline{a(v,u)}
$$
for $u,v \in D(a^{\ast}) = D(a)$.
\end{enumerate}
\end{dfn}

Now, let the form $a \colon D(a) \times D(a) \to \mathbb{K}$ be densely defined, accretive, coninuous and closed. 
Then we define a subspace
$$
D(A) = \left\{ u \in D(a) \ | \ \exists v \in H : a(u, \varphi) = \langle v, \varphi \rangle \ \text{ holds for every } \varphi \in D(a) \right\}
$$ 
as the domain of an unbounded operator $A \colon D(A) \subset H \to H$ defined by the equation
$$
a(u, \varphi) =  \langle Au, \varphi \rangle
$$
for $u \in D(A)$ and every $\varphi \in D(a)$. We call $A$ the associated operator to the form $a$.\\

\begin{thm}\label{generator_semigroup}
Let $A$ be the associated operator to the form $a$. Then $D(A)$ is dense in $H$ and for every $\lambda > 0$ the operator
$\lambda + A$ is invertible such that
$$
\left\| \lambda \left( \lambda + A \right)^{-1} u \right\| \leq \| u \|
$$
holds for every $u \in H$. Hence $-A$ is the generator of a $C_{0}$-semigroup $\left( \mathrm{e}^{-tA} \right)_{t \geq 0}$
of contractions in $H$.\\
\end{thm}

\begin{dfn}
Let $A \colon D(A) \subset H \to H$ be a densely defined operator in $H$. Then the adjoint operator $A^{\ast}$ to $A$ is defined
by
$$
D(A^{\ast}) = \left\{ u \in H \ | \ \exists v \in H : \langle A\varphi, u \rangle = \langle \varphi, v \rangle \ \text{ for every } 
\varphi \in D(A) \right\}
$$  
and $\langle A\varphi, u \rangle = \langle \varphi, A^{\ast}u \rangle$ for $u \in D(A^{\ast})$ and every $\varphi \in D(A)$.\\
\end{dfn}

\begin{lem}\label{generator_adjoint_semigroup}
The adjoint operator $A^{\ast}$ of $A$ is the associated operator to the adjoint form $a^{\ast}$ of $a$. Furthermore
 $- A^{\ast}$ is the generator of a $C_{0}$-semigroup $\left( \mathrm{e}^{-tA^{\ast}} \right)_{t \geq 0}$ in $H$ such that 
$$
\mathrm{e}^{-tA^{\ast}} = \left( \mathrm{e}^{-tA} \right)^{\ast}
$$
holds for every time $t \geq 0$.
\end{lem}

\section{Nash's inequality on $H^{1}( \Omega )$}

Let $\Omega \subset \R^{d}$ for $d > 2$ be a bounded Lipschitz domain. We will present the details in the argumentation of Nash's inequality which reads
\begin{equation} \label{Nash_inequality}
\exists C_{d} > 0 \ \forall u \in H^{1}( \Omega ) \ : \
\| u \|_{\mathrm{L}^{2}(\Omega)}^{2+\frac{4}{d}} \leq C_{d} \| u \|_{\mathrm{L}^{1}(\Omega)}^{\frac{4}{d}} 
\| \nabla u \|_{\mathrm{L}^{2}(\Omega)^{d}}^{2}
\end{equation}
where we keep in mind that $H^{1}( \Omega ) \subset \mathrm{L}^{2} (\Omega)$ is also contained in 
$\mathrm{L}^{1} (\Omega)$ since $\Omega$ is bounded. We start to prove (\ref{Nash_inequality}) for 
$u \in C^{1}( \Omega )$ and use Hölder inequalities to conclude to
$$
\int_{\Omega} |u|^{2} \d x = \int_{\Omega} |u|^{\frac{4}{d+2}} |u|^{\frac{2d}{d+2}} \d x 
\leq \left( \int_{\Omega} |u| \d x \right)^{\frac{4}{d+2}} \left( \int_{\Omega} |u|^{\frac{2d}{d-2}} \d x \right)^{\frac{d-2}{d+2}}.
$$
Due to $\emptyset \neq \Omega \subset \R^{d}$ being a Lipschitz domain, we use a generalization of the 
Sobolev embedding theorem to infer that 
$$
\left( \int_{\Omega} |u|^{\frac{2d}{d-2}} \d x \right)^{\frac{d-2}{2d}} \leq C_{d} 
\left( \int_{\Omega} |\nabla u|^{2} \d x \right)^{\frac{1}{2}}
$$
is true where $C_{d} > 0$ is a constant independent of $u$. We conclude to  
$$
\int_{\Omega} |u|^{2} \d x \leq C_{d}^{\frac{2d}{d+2}} \left( \int_{\Omega} |u| \d x \right)^{\frac{4}{d+2}} 
\left( \int_{\Omega} |\nabla u|^{2} \d x \right)^{\frac{d}{d+2}}
$$
which implies 
$\| u \|_{\mathrm{L}^{2}(\Omega)}^{2+\frac{4}{d}} \leq \tilde{C}_{d} \| u \|_{\mathrm{L}^{1}(\Omega)}^{\frac{4}{d}} 
\| \nabla u \|_{\mathrm{L}^{2}(\Omega)^{d}}^{2}$
for $\tilde{C}_{d} = C_{d}^{2} > 0$. Due to a density argumentation we have shown (\ref{Nash_inequality}) on $H^{1}( \Omega )$.\\

\newpage

\bibliography{references}
\bibliographystyle{plain}

\end{document}